\documentclass{birkjour}
\usepackage{amscd, amssymb, pgf, tikz, hyperref}

\newcommand{\field}[1]{\mathbb{#1}} 

\newcommand{\NN}{\field{N}}

\newcommand{\ZZ}{\field{Z}}

\newcommand{\Aa}{\mathcal{A}}
\newcommand{\Bb}{\mathcal{B}}

\newcommand{\Ff}{\mathcal{F}}
\newcommand{\Gg}{\mathcal{G}}
\newcommand{\Hh}{\mathcal{H}}

\newcommand{\Mm}{\mathcal{M}}

\newcommand{\Oo}{\mathcal{O}}

\newcommand{\Ss}{\mathcal{S}}

\newcommand{\la}{\langle}
\newcommand{\ra}{\rangle}

\newcommand{\ds}{\displaystyle}

\newcommand{\go}{\Gg^{(0)}}
\newcommand{\Bis}{\operatorname{Bis}}
\newcommand{\ob}{\overline{b}}
\newcommand{\od}{\overline{\partial}}

\newtheorem{thm}{Theorem}[section]
\newtheorem{cor}[thm]{Corollary}
\newtheorem{lem}[thm]{Lemma}
\newtheorem{prop}[thm]{Proposition}

\theoremstyle{definition}
\newtheorem{dfn}[thm]{Definition}

\theoremstyle{remark}
\newtheorem{rmk}[thm]{Remark}

\newtheorem{example}[thm]{Example}

\newtheorem*{examples*}{Examples}

\numberwithin{equation}{section}

\title{ Cohomology of ample groupoids }

\author{Valentin Deaconu}
\address{Department of Mathematics (084)\\ University
of Nevada\\ Reno NV 89557-0084\\ USA} \email{vdeaconu@unr.edu}
\author{Marius Ionescu}
\address{United States Naval Academy\\ Annapolis\\ MD 21402-5002\\ USA}
\email{ionescu@usna.edu}


\thanks{The views expressed in this paper are those of the authors 
and do not reflect the official policy or position of the U.S. Naval
Academy, Department of the Navy, the Department of Defense, or the
U.S. 
Government.}
\begin{document}
\begin{abstract}
We introduce a  cochain complex for  ample groupoids $\Gg$ using a flat resolution defining their homology with  coefficients in $\mathbb Z$. We prove that the cohomology of this cochain complex with values in a $\Gg$-module $M$ coincides  with the previously introduced  continuous cocycle cohomology of $\Gg$. In particular, this groupoid cohomology is invariant under Morita equivalence. We derive an exact sequence for the cohomology of skew products by a $\ZZ$-valued cocycle. We indicate how to compute the cohomology with coefficients in a $\Gg$-module $M$ for $AF$-groupoids and for certain action groupoids.  
\end{abstract}
\maketitle
\section{Introduction}

\bigskip

The cohomology of \'etale groupoids was first defined by Haefliger in Chapter III of \cite{H76}, using a complex of non homogeneous cochains with values in a sheaf.  In his thesis \cite{R80}, Renault defines the cohomology of a topological groupoid by using (normalized) continuous cocycles with values in a locally compact group bundle.  To define cohomology groups in connection with elementary $C^*$-bundles, Kumjian  is using sheaves and derived functors of the invariant section functor, see \cite{K88}. For a review of some of these definitions and the connection with the cohomology of small categories, see \cite{GK}. In \cite{T06},  Tu showed  that Haefliger's cohomology for \'etale groupoids, Moore's cohomology for locally compact groups and the Brauer group of a locally compact groupoid are particular cases of sheaf cohomology for topological simplicial spaces.

Recently, there has been significant progress in understanding the homology of ample groupoids $\Gg$ and their relationship with the $K$-theory of their $C^*$-algebra $C^*(\Gg)$, see for example \cite{M12, FKPS, BDGW, PY22}. There are connections with the dynamic asymptotic dimension of an \'etale groupoid and, using the unstable equivalence relation of a Smale space, with the homology of hyperbolic dynamical systems defined by I. Putnam, see \cite {PY23}. One key ingredient for ample groupoids is the fact that the category of $\Gg$-sheaves is equivalent to the category of $\Gg$-modules, see \cite{S14}. In a recent paper \cite{Li25}, X. Li constructed a spectrum whose homology groups recover groupoid homology, and proved the  $AH$ conjecture of Matui, in connection with the topological full group. Many of these results are proven for not necessarily Hausdorff groupoids. In this paper, all spaces and all groupoids that we consider are Hausdorff.

In the hope to facilitate concrete computations of cohomology groups, we dualize a resolution used for the homology  of ample groupoids, which appears  in \cite{M22}, see also \cite{CM} and Matui \cite{M12}. 

We begin with a review of the homology of ample groupoids and of $\Gg$-modules $M$ in section 2. In section 3, we define the cohomology groups $H^n(\Gg, M)$ using a dual complex. This section is using the equivalence of $\Gg$-sheaves and $\Gg$-modules for ample groupoids. Our first main result is Theorem \ref{thm:main}, where we prove that the cohomology  with values in a $\Gg$-module is isomorphic with the cohomology defined using cocycles. The main inspiration was the paper by Gillaspy and Kumjian \cite{GK}, where they work with sheaves instead of modules. 

In section 4, we  prove an exact sequence for computing the cohomology of skew products of ample groupoids by a $\ZZ$-valued cocycle. In section 5, we illustrate the theory with several examples, like the computation of cohomology for $AF$-groupoids and for certain action groupoids. 
 
 Recently, a preprint of Matui and Mori \cite{MM} explores the ring structure using the cup product on groupoid cohomology with integer coefficients  and the cap product between homology and cohomology. In their definition of cohomology, they use cocycles with values in an abelian group and the groupoid action is trivial.
 
 We hope that this paper will stimulate further research and connections with dynamical systems and with invariants of $C^*$-algebras.
 
 {\bf Acknowledgement}. We thank the referees for very detailed suggestions that helped to improve the quality of the paper. We also thank Alex Kumjian for helpful discussions.
\bigskip

\section{Homology  of ample groupoids and $\Gg$-modules}

\bigskip

In this section, we review the definition of homology of a groupoid $\Gg$ and of the concept of $\Gg$-module. We first recall  the definition of homology of ample (Hausdorff) groupoids which was introduced  in \cite{CM} in a more general framework,  and studied in \cite{M12} for the case of ample groupoids. Recall that  an ample groupoid $\Gg$ is an \'etale  groupoid such that its unit space 
$\Gg^{(0)}$ is totally disconnected. 

Let $A$ be a topological abelian group and
let $\pi: X \to Y$ be a local homeomorphism between two locally compact Hausdorff spaces. Denote by $C_c(X, A)$ the abelian group of continuous compactly supported functions with pointwise addition. Given $f\in C_c(X, A)$, define a map
\begin{equation}
  \label{eq:pi_star} 
  \pi_*:C_c(X,A)\to C_c(Y,A),\;\; \pi_*(f)(y):= \sum_{\pi(x)=y} f(x),
\end{equation}
which is a group homomorphism.

For  an \' etale groupoid $\Gg$, let $\Gg^{(1)}=\Gg$ and for $n\ge 2$, let $\Gg^{(n)}$ be the space of composable strings $(g_1,g_2,...,g_n)$
of $n$ elements in $\Gg$, with the product topology. For $n\ge 2$ and  $ i = 0,...,n$, we let $\partial^n_i : \Gg^{(n)} \to \Gg^{(n-1)}$ be the face maps defined by

\[\partial^n_i(g_1,g_2,...,g_n)=\begin{cases}(g_2,g_3,...,g_n)&\;\text{if}\; i=0,\\
 (g_1,...,g_ig_{i+1},...,g_n) &\;\text{if} \;1\le i\le n-1\\
(g_1,g_2,...,g_{n-1})& \;\text{if}\; i = n.\end{cases},\]
which are local homeomorphisms.
Consider the homomorphisms of abelian groups $d_n : C_c(\Gg^{(n)}, A) \to C_c(\Gg^{(n-1)}, A)$ given by 
\begin{equation}
\label {diff}
d_1=s_*-r_*,\;\; d_n=\sum_{i=0}^n(-1)^i\partial^n_{i*}\; \text{for}\; n\ge 2.
\end{equation}
Recall that $s,r: \Gg\to \Gg^{(0)}$ are the source and the range maps and \[\partial^n_{i*}: C_c(\Gg^{(n)},A)\to C_c(\Gg^{(n-1)},A), \]\[ \partial^n_{i*}(f)(g_1,...,g_{n-1})=\sum_{\partial^n_i(h_1,h_2,...,h_n)=(g_1,...,g_{n-1})}f(h_1,h_2,...,h_n).\]
It can be verified that the differentials $d_n$ satisfy $d_n\circ d_{n+1}=0$ for all $n\ge 1$.

The homology groups $H_n(\Gg, A)$ are by definition  the homology groups of the chain complex $C_c(\Gg^{(*)},A)$ given by
\[
0\stackrel{d_0}{\longleftarrow} C_c(\Gg^{(0)},A)\stackrel{d_1}{\longleftarrow}C_c(\Gg^{(1)},A)\stackrel{d_2}{\longleftarrow}C_c(\Gg^{(2)},A)\longleftarrow\cdots,\]
i.e. $H_n(\Gg,A)=\ker d_n/\text{im } d_{n+1}$. We write $H_n(\Gg)$ for $H_n(\Gg,\ZZ)$.

When  $\Gg$ is a discrete group $G$, the above chain complex coincides with the standard bar complex and  $H_*( G, A)$ recovers the group homology with coefficients in $A$, where $A$ becomes a $G$-module with trivial action (see \cite{Br} for example). 

It is known that two Morita equivalent \'etale groupoids have the same homology, see section 3 in \cite{CM}. For different kinds of equivalence of groupoids, including Kakutani equivalence and similarity of groupoids, see section 3 in \cite{FKPS}. In particular, the homology of a proper principal groupoid is isomorphic to the homology of the orbit space. 

An \'etale  groupoid homomorphism $\phi:\Gg_1\to \Gg_2$ induces local homeomorphisms 
\[\phi^{(n)}:\Gg_1^{(n)}\to \Gg_2^{(n)},\;\; \phi^{(n)}(g_1,...,g_n)=(\phi(g_1),...,\phi(g_n))\] and maps $\phi_*^{(n)}:C_c(\Gg_1^{(n)},A)\to C_c(\Gg_2^{(n)},A)$, 
\begin{equation}
\label{eq phi}
\phi_*^{(n)}(f)(h_1,...,h_n)=\sum_{\phi^{(n)}(g_1,...,g_n)=(h_1,..., h_n)}f(g_1,..., g_n)
\end{equation}
for $(g_1,...,g_n)\in \Gg_1^{(n)}$ and $(h_1,...,h_n)\in \Gg_2^{(n)}$, which commute with the differentials. Therefore, being a local homeomorphism, the homomorphism $\phi$ induces  homology group homomorphisms, denoted $\phi_*:H_*(\Gg_1,A)\to H_*(\Gg_2,A)$, and $\phi\mapsto \phi_*$ preserves composition.  As a consequence, if $\{\Gg_n\}_{n\ge 1}$ is an increasing sequence of open subgroupoids of $\Gg$ such that $\ds \Gg=\bigcup_{n=1}^\infty \Gg_n$, then $H_*(\Gg,A)\cong \varinjlim H_*(\Gg_n,A)$. For example, the homology of an $AF$-groupoid can be computed using inductive limits. 

 We write $\Bis(\Gg)$ for the set
of compact open bisections of an ample groupoid. 
Since an ample  groupoid has  a basis of
compact open bisections,  the (free) abelian group $C_c(\Gg,\ZZ)$ consists of
locally constant functions with compact open support. It is generated
by the indicator functions $\chi_U$ of compact open bisections. In
\cite{BDGW}, \cite{M23} and in other papers, $C_c(\Gg,\ZZ)$ is denoted by
$\ZZ[\Gg]$  and it has a ring structure with multiplication given by
convolution: for $f_1,f_2\in \ZZ[\Gg]$, 
\[(f_1 f_2)(g)=\sum_{h\in \Gg_{r(g)}}f_1(h^{-1})f_2(hg),\] where $\Gg_u=\{g\in \Gg: s(g)=u\}$. This ring has local units, in the sense that for any finite collection $f_1,...,f_n$ of elements in $\ZZ[\Gg]$, there is an idempotent $e\in \ZZ[\Gg]$ such that $e f_i=f_i e=f_i$ for each $i=1,...,n$. One can take $e=\chi_U$ for a certain compact open set $U\subseteq \Gg^{(0)}$. 

\begin{dfn}For $\Gg$ an ample groupoid, a $\Gg$-module is a (left) $\ZZ[\Gg]$-module $M$ (assumed non-degenerate in the sense that $\ZZ[\Gg]M=M$). 
\end{dfn}

\begin{dfn}\label{sa}  A topological groupoid $\Gg$ is said to act
(on the left) on a locally compact space $X$, if there are given 
a continuous  surjection $p : X \rightarrow \Gg^{(0)}$,  called the anchor or moment map,
and a continuous map
\[\Gg\ast X \rightarrow X, \quad\text{write}\quad (g , x)\mapsto
g \cdot x=gx,\]
where
\[\Gg \ast X = \{(g , x)\in \Gg \times X \mid s(g) =p (x)\},\]
that satisfy

\medskip

i) $p (g \cdot x) =r(g)$ for all $(g , x) \in \Gg \ast X,$

\medskip

ii) $(g _2, x) \in \Gg \ast X,\,\, (g_1, g_2)
\in \Gg ^{(2)}$ implies $(g _1g _2, x),
(g _1, g _2\cdot x) \in \Gg * X$ and
\[g _1\cdot(g _2\cdot x) = (g _1g _2)
\cdot x,\]

\medskip

iii) $p (x)\cdot x = x$ for all $x\in X$.

\end{dfn}

The  locally compact Hausdorff spaces $X$ on which an ample groupoid $\Gg$ acts  such that the anchor map $p:X\to \Gg^{(0)}$ is a local homeomorphism are also called $\Gg$-sheaves of sets (see Definition 3.9 from \cite{GK}) or \'etale $\Gg$-spaces.  They provide important examples of $\Gg$-modules, with the left action given by
\[(f\cdot m)(x)=\sum_{g\in \Gg_{p(x)}}f(g^{-1})m(g\cdot x)\] for $f\in \ZZ[\Gg]$ and $m\in \ZZ[X]=C_c(X,\ZZ)$. Note that $X$ is also totally disconnected.

We will use the following construction and lemma later. Let $X$ be a
left $\Gg$-space such that the anchor map $p$ is a local 
homeomorphism. For $x\in X$ and $V$  a compact open subset of
$X$ such that $x\in V$ and $p|_V:V\to p(V)$ is a homeomorphism, we
define the element $\langle x\rangle_V\in C_c(X,\ZZ)$ to be the
indicator function of $V$:
\begin{equation}
  \label{eq:x_V}
  \langle  x \rangle_V(y)=\chi_V(y)=
  \begin{cases}
    1 & \text{ if }y\in V\\
    0 & \text{ otherwise.}
  \end{cases}
\end{equation}
Note that $\la x\ra_V(x)=1$.

\begin{lem}\label{lem:section_xV}
  Assume that $X$ is a left $\Gg$-space such that the anchor map $p$ is a local 
homeomorphism. Let $x\in X$ and let $V$ be as in the paragraph preceding this
lemma. Let $g\in \Gg$ such that $s(g)=p(x)$ and let $U\in \Bis(\Gg)$
such that $g\in U$
. Then
\begin{equation}
  \label{eq:g_x_V}
  \chi_U\cdot \langle x \rangle_V=\langle  g\cdot x \rangle_{UV},
\end{equation}
where $UV=\{h\cdot y\,:\,h\in U,y\in V,s(h)=p(y)\}$.
\end{lem}
\begin{proof}
  First note that $UV$ is a compact open subset of $X$ and the
  restriction of $p$ to $UV$ is a homeomorphism onto $p(UV)$.
  
  For $y\in X$ we have
  \[
    \chi_U\cdot \langle  x
    \rangle_V(y)=\sum_{s(h)=p(y)}\chi_U(h^{-1})\langle x \rangle_V(h\cdot
    y).
  \]
  Therefore $\chi_U\cdot \langle  x \rangle_V(y)=1$ if and only if there
  are (unique) $h\in U^{-1}$ and $z\in V$ such that $y=h^{-1}z$. The
  conclusion follows.
\end{proof}

Since $\Gg$ acts on the left on $\Gg^{(n)}$ using the anchor map \[p:
  \Gg^{(n)}\to \Gg^{(0)}, \;\;p(g_1, g_2,...,g_n)=r(g_1)\] such
that  \[g\cdot(g_1,g_2,...,g_n)=(gg_1,g_2,...,g_n)\] for
$s(g)=r(g_1)$, 
the abelian groups $\ZZ[\Gg^{(n)}]$ become  $\Gg$-modules in a natural way. 

\begin{rmk}
To define the homology of an ample groupoid with values in a $\Gg$-module $M$, Miller (see Example 2.14 in \cite{M23} or Chapter 4 in \cite{M22})  is using a flat resolution of the $\Gg$-module $\ZZ[\Gg^{(0)}]$ as in \cite{Li25}, called the bar resolution:
\begin{equation}\label{bar}
\cdots\stackrel{b_{n+1}}{\longrightarrow}\ZZ[\Gg^{(n+1)}]\stackrel{b_n}{\longrightarrow}\ZZ[\Gg^{(n)}]\stackrel{b_{n-1}}{\longrightarrow}\cdots\stackrel{b_1}{\longrightarrow} \ZZ[\Gg^{(1)}]\stackrel{b_0}{\longrightarrow}\ZZ[\Gg^{(0)}]\to 0.
\end{equation}
This resolution is in fact projective when the unit space of $\Gg$ is $\sigma$-compact and Hausdorff, see \cite{BDGW}.
 For $n\ge 1$ and $0\le i\le n$, let $b_i^n:\Gg^{(n+1)}\to \Gg^{(n)}$ be such that
\[b_i^n(g_0,...,g_n)=\begin{cases}(g_0,...,g_ig_{i+1},...,g_n)&\;\text{if}\; i<n,\\
 (g_0,...,g_{n-1})& \;\text{if}\; i = n.\end{cases}\]
 The maps $b_i^n$ are $\Gg$-equivariant local homeomorphisms and induce $\Gg$-module maps $b^n_{i*}:\ZZ[\Gg^{(n+1)}]\to \ZZ[\Gg^{(n)}]$. Then  $b_n:\ZZ[\Gg^{(n+1)}]\to \ZZ[\Gg^{(n)}]$ for $n\ge 1$ are given by 
\begin{equation}\label {bardiff}
b_n=\sum_{i=0}^n(-1)^ib^n_{i*}
\end{equation}
and let $b_0=s_*:\ZZ[\Gg^{(1)}]\to\ZZ[\Gg^{(0)}]$.
 The exactness of the bar resolution \eqref{bar} is witnessed by a chain homotopy induced by local homeomorphisms
  \[h_n:\Gg^{(n)}\to \Gg^{(n+1)},\; h_n(g_0,...,g_{n-1})=(r(g_0),g_0,...,g_{n-1})\;\text{for}\; n\ge 1,\]
  with $h_0:\Gg^{(0)}\to \Gg$ being the inclusion. 
  
 The coinvariants of a $\Gg$-module $M$ is the abelian group $M_\Gg=\ZZ[\Gg^{(0)}]\otimes_\Gg M$. The coinvariants of the $\Gg$-module $\ZZ[\Gg^{(n+1)}]$ for $n\ge 1$ is isomorphic to $\ZZ[\Gg^{(n)}]$. 
 Taking the coinvariants of the above bar resolution, one obtains the new chain complex
  \[\cdots\stackrel{(b_{n+1})_{\Gg}}{\longrightarrow}\ZZ[\Gg^{(n)}]\stackrel{(b_n)_{\Gg}}{\longrightarrow}\ZZ[\Gg^{(n-1)}]\longrightarrow\cdots\stackrel{(b_2)_{\Gg}}{\longrightarrow} \ZZ[\Gg^{(1)}]\stackrel{(b_1)_{\Gg}}{\longrightarrow}\ZZ[\Gg^{(0)}]\to 0,\]
  where $(b_n)_{\Gg}$ are in fact the differentials $d_n$ as in \eqref{diff} defined using the face maps $\partial^n_i:\Gg^{(n)}\to \Gg^{(n-1)}$ for $n\ge 2$. The homology of this new chain complex computes $H_*(\Gg,\ZZ)\cong\text{Tor}^{\Gg}_*(\ZZ[\Gg^{(0)}], \ZZ[\Gg^{(0)}])$. A similar resolution of $M$ can be used to compute $H_*(\Gg, M)\cong\text{Tor}^{\Gg}_*(\ZZ[\Gg^{(0)}],M)$.
\end{rmk}
\bigskip

 \section{ Cohomology of ample groupoids}
 \bigskip

In this section, we obtain the first main result, relating the
cohomology of an ample groupoid defined using a cochain complex with
the cocycle cohomology. Many facts are just a reinterpretation of
results in \cite{GK} from the context of sheaves to the context of
modules using the equivalence between $\Gg$-sheaves and $\Gg$-modules
for ample groupoids proved in \cite{S14}. While certain results in
this section could be derived from the previously cited works, we
present complete proofs by tailoring the general theory to our
specific case. Our primary motivation for this approach is to offer
concrete formulas that directly facilitate the cohomology computations
for the examples that we analyze. 

\begin{dfn}
Let $\Gg$ be an ample groupoid and let $M$ be any
$\Gg$-module. Consider the dual complex
Hom$_{\Gg}(\ZZ[\Gg^{(*)}],M)=\text{Hom}_{\ZZ[\Gg]}(\ZZ[\Gg^{(*)}],M)$
with (co)differentials \[\delta_n:\text
  {Hom}_{\Gg}(\ZZ[\Gg^{(n+1)}],M)\to
  \text{Hom}_{\Gg}(\ZZ[\Gg^{(n+2)}],M),
  \delta_n(\varphi)=\varphi\circ b_{n+1}\] for $\varphi\in \text
  {Hom}_{\Gg}(\ZZ[\Gg^{(n+1)}],M)$, where $b_n :
\ZZ[\Gg^{(n+1)}] \to \ZZ[\Gg^{(n)}]$ are the differentials
defined in \eqref{bardiff}. We define the cohomology groups $H^n(\Gg,M)$
as the cohomology of this dual complex,
i.e. $H^n(\Gg,M)=\ker\delta_n/\text{im }\delta_{n-1}$. 
\end{dfn}

We will see below how this groupoid cohomology relates to previous
versions of cohomology. The main result of this section, Theorem
\ref{thm:main}, proves that our cohomology with coefficients in a $\Gg$-module $M$ is isomorphic with the
cohomology defined using  continuous cocycles with values in a specific
$\Gg$-sheaf $\Mm$ of abelian groups.

Recall that a sheaf of abelian groups over a   space $X$ is a
topological space $\Aa$ with a local homeomorphism $\pi:\Aa\to X$ such
that each fiber $\Aa_x=\pi^{-1}(x)$ is an abelian group and the group
operations are continuous. If $\Gg$ is an \'etale 
groupoid, a $\Gg$-sheaf is a sheaf $\Aa$ over $\Gg^{(0)}$ such that
for each $g\in \Gg$ there are isomorphisms $\alpha_g:\Aa_{s(g)}\to
\Aa_{r(g)}$ with the properties 
\[x\in \Gg^{(0)}\Rightarrow \alpha_x=\text{id},\;\;(g_1,g_2)\in
  \Gg^{(2)}\Rightarrow
  \alpha_{g_1}\circ\alpha_{g_2}=\alpha_{g_1g_2},\]\[\alpha:\Gg\ast\Aa\to\Aa,\;
  (g,a)\mapsto \alpha_g(a)\;\text{is continuous}.\] 
We write $g\cdot a$ for $\alpha_g(a)$ as is customary.

\begin{rmk}\label{sheaf-mod}
In Theorem 3.5 of \cite{S14} it is proved that for (not necessarily Hausdorff) ample groupoids,
the category of (right) $\Gg$-sheaves is equivalent to the category of (right)
non-degenerate $\Gg$-modules. We choose to consider left $\Gg$-modules and left $\Gg$-sheaves with $\Gg$ Hausdorff, so we adapt the formulas accordingly. 

Specifically, given a $\Gg$-sheaf $\Aa$ with $\pi:\Aa\to \Gg^{(0)}$,
the space $\Gamma_c(\Aa,\pi)$ of compactly supported continuous
sections $\xi:\Gg^{(0)}\to \Aa$ becomes a $\Gg$-module using 
\[(f\xi)(x)=\sum_{r(g)=x} f(g)(g\cdot \xi(s(g)))\]
for $f\in \ZZ[\Gg]$. Conversely, any $\Gg$-module $M$ determines a $\Gg$-sheaf $\Mm$ by using the compact open subsets $U$ of $\Gg^{(0)}$ to define the fibers (or germs) $\ds M_x=\varinjlim_{x\in U} \chi_UM$ and then $\ds \Mm=\bigsqcup_{x\in \Gg^{(0)}}M_x$ becomes a $\Gg$-sheaf with appropriate topology and $\Gg$-action. A basis for the topology on $\Mm$ is given by the sets
\[(U,m)=\{[m]_x:x\in U\},\]
where $U\subseteq \Gg^{(0)}$ is compact open and $[m]_x\in M_x$ denotes the image of $m\in \chi_UM$ in the inductive limit. The $\Gg$-action is defined by \[g\cdot [m]_{s(g)}=[\chi_Vm]_{r(g)},\] where $V$  is a compact open bisection with $g\in V$.
Moreover, the proof of
\cite[Theorem 3.5]{S14} implies that the module $M$ is isomorphic with
the module $\Gamma_c(\Mm,\pi)$ via the isomorphism $\eta_M:M\to
\Gamma_c(\Mm,\pi)$, $\eta_M(m)=s_m$, where $\pi:\Mm\to \Gg^{(0)}$ is
the projection and $s_m(x)=[m]_x$ for all $x\in \Gg^{(0)}$. We will
describe the isomorphism between the $\Gg$-sheaf morphisms and
$\Gg$-module morphisms that we study in Proposition \ref{lem:hom_isom}.

\end{rmk}

\begin{rmk}

Let $\Gg$ be any \'etale groupoid and let $p:Y\to \go$ be an \'etale
$\Gg$-space. It is proven in \cite{GK} that 
there is a $\Gg$-sheaf denoted $\ZZ[Y]$  with the stalk at $x\in
\Gg^{(0)}$ given by the free abelian group $\ZZ[Y_x]$ generated by the
fiber $Y_x:=p^{-1}(x)$ and the topology as described in \cite[\S 3.1]{GK}. 
For any  \'etale groupoid $\Gg$, in particular for any ample groupoid, note that
$\Gg^{(n)}$ is  a $\Gg$-sheaf of sets. 

The notation from \cite{GK} is related  to our notation, but
unfortunately is not the same. To distinguish between the $\Gg$-module
$\ZZ[Y]=C_c(Y,\ZZ)$ and the $\Gg$-sheaf $\ZZ[Y]$ from \cite{GK}, we
will use $\ZZ[Y]^s$ for the latter. We write $p^s$ for the
projection of $\ZZ[Y]^s$ onto $\go$.

\end{rmk}

The following lemma provides a concrete presentation of Steinberg's
construction (\cite{S14}) as reviewed in Remark \ref{sheaf-mod}
applied to the
$\Gg$-sheaf
$\ZZ[Y]^s$ for any \'etale $\Gg$-space $Y$ and, in particular, for
$Y=\Gg^{(n)}$, where $p: \Gg^{(n)}\to \Gg^{(0)}, \;\;p(g_1,
g_2,...,g_n)=r(g_1)$. Specifically, the lemma  identifies  
the module of sections associated to the $\Gg$-sheaf
$\ZZ[Y]^s$ with the $\Gg$-module $\ZZ[Y]$.

\begin{lem}\label{lem:isom_gamma_n}
  Assume that $\Gg$ is an ample groupoid and $p:Y\to \go$ is an \'etale $\Gg$-space.
  The map $\Phi:\ZZ[Y]\to \Gamma_c(\ZZ[Y]^s,p^s)$
  defined via
  \[
  \Phi(m)(x)=\sum_{p(y)=x}m(y)[y],
  \]
  for all $m\in \ZZ[Y]$ is an isomorphism of $\Gg$-modules, where $[y]$ is the
  generator determined by $y\in Y$ in the free abelian group
  $\ZZ[Y_x]$.
\end{lem}
\begin{proof}
  To see that $\Phi$ is a bijection, we will define its inverse. Let
  $\xi\in \Gamma_c(\ZZ[Y]^s,p^s)$. By definition, if $x\in \Gg^{(0)}$, there exist
  finitely many non-zero $a_{y}\in \ZZ$ with $y\in Y_x$
 such that
  \[
    \xi(x)=\sum_{p(y)=x}a_{y}[y].
  \]
  Then we take $\Phi^{-1}(\xi)(y)=a_{y}$. It is easy to see that
  $\Phi^{-1}\circ \Phi$ and $\Phi\circ \Phi^{-1}$ are the
  identity maps. 

  We check next that $\Phi$ is a module morphism. Let $f\in\ZZ[\Gg]$ and
  $m\in \ZZ[Y]$. Then
  \begin{multline*}
\Phi(f\cdot m)(x)=\sum_{p(y)=x}(f\cdot
   m)(y)[y]
    =\sum_{p(y)=x}\sum_{s(g)=x}f(g^{-1})m(gy)[g^{-1}gy]\\
 =\sum_{r(h)=x}f(h) h\cdot\left(\sum_{p(y)=s(h)}m(y)[y]\right)=f\cdot \Phi(m)(x)    
\end{multline*}
for all $x\in \Gg^{(0)}$,  where the last equality follows from the
previous line by relabeling $g^{-1}$ with $h$ and  $gy$ with $y$.

\end{proof}

\begin{rmk}\label{rm:Phi_n}
  If $Y=\Gg^{(n)}$, we write $\Phi_n$ for the map provided by the lemma:
  $\Phi_n:\ZZ[\Gg^{(n)}]\to \Gamma_c(\ZZ[\Gg^{(n)}]^s,p^s)$ defined via
  \[
    \Phi_n(m)(x)=\sum_{p(g_1,\dots,g_n)=x}m(g_1,\dots,g_n)[g_1,\dots,g_n],
  \]
  where $[g_1,\dots, g_n]$ is the generator determined by $(g_1,\dots, g_n)$ in the free abelian group $\ZZ[\Gg_x^{(n)}]$.
\end{rmk}
 
\begin{rmk}
  If $Y$ is an \'etale $\Gg$-space, $y\in Y$ and $V$ is a compact open
  subset of $Y$, we let $\langle y \rangle_V^s$ be the image under the
  map $\Phi$ of the section $\langle  y \rangle_V$ defined in
  \eqref{eq:x_V}. In particular, if $p|_V$ is a homeomorphism onto
  $p(V)$ it follows that
  \begin{equation}
    \label{eq:y_V^s}
    \langle y  \rangle_V^s(x)=
    \begin{cases}
      [z] & \text{ if }z\in V\text{ and }p(z)=x\in p(V)\\
      0 & \text{ otherwise.}
    \end{cases}
  \end{equation}
  Hence $\langle y \rangle_V^s(y)=[y]$.
\end{rmk}


The above lemma allows us to prove that our definition of cohomology
recovers the sheaf cohomology as defined \cite[Chapter III]{H76} and
\cite[Chapter I]{R80}. We follow the notation of   \cite[\S 2
and \S3]{GK}. 
We recall the definition of continuous cocycle sheaf cohomology:

\begin{dfn}
Let $\Gg$ be an \'etale groupoid and let $\Aa$ be a $\Gg$-sheaf. The set of continuous $n$-cochains with values in $\Aa$ is
\[C^n(\Gg,\Aa)=\{f:\Gg^{(n)}\to \Aa\;\mid\;f\;\text{continuous,}\; f(g_1,...,g_n)\in \Aa_{r(g_1)}\}.\]
The differentials (or boundary maps) are defined for $n\ge 1$ by
\[\delta_c^n:C^n(\Gg, \Aa)\to C^{n+1}(\Gg,
  \Aa),\;(\delta_c^nf)(g_0,g_1,...,g_n)=\]
\[=g_0\cdot f(g_1,...,g_n)+\sum_{i=1}^n(-1)^if(g_0,...,g_{i-1}g_i,...,g_n)+(-1)^{n+1}f(g_0,...,g_{n-1}),\]
and for $n=0$ by $(\delta_c^0f)(g_0)=g_0f(s(g_0))-f(r(g_0))$.
The continuous cocycle sheaf cohomology is defined as $H_c^n(\Gg, \Aa)=(\ker \delta_c^n)/( \text{im } \delta_c^{n-1})$ with $\delta_c^{-1}=0$.
\end{dfn}

\begin{prop}\label{prop:sheaf_coh}
  For $n\ge 1$ let $\od_n:\ZZ[\Gg^{(n+1)}]^s\to\ZZ[\Gg^{(n)}]^s$ be
  defined via
  \[
    \od_n([h_0,\dots,h_n])=\sum_{i=0}^{n-1}(-1)^i[h_0,\dots,h_ih_{i+1},\dots ,
    h_n]+ (-1)^n[h_0,\dots,h_{n-1}].
  \]
  For $n=0$ we take $\od_0:\ZZ[\Gg^{(1)}]^s\to \ZZ[\Gg^{(0)}]^s, \;\od_0([h_0])=[s(h_0)]$. Define $\ob_n:\Gamma_c(\ZZ[\Gg^{(n+1)}]^s,p^s)\to
  \Gamma_c(\ZZ[\Gg^{(n)}]^s,p^s)$ via
  $\ob_n(\xi)(x)=\od_n(\xi(x))$. Then
  \begin{equation}\label{eq:intertwine_bn}
    \Phi_n\circ b_n=\ob_n\circ \Phi_{n+1},
 \end{equation}
     where $b_n$ were defined in \eqref{bardiff}.
\end{prop}
\begin{proof} Let $n\ge 1$.
  If $\;0\le i<n$ and $(g_0,\dots,g_{n-1})\in \Gg^{(n)}$, then
$b_i^n(h_0,\dots,h_n)=(g_0,\dots,g_{n-1})$ implies that $h_j=g_j$ for
  all $j<i$, $h_ih_{i+1}=g_i$, and $h_j=g_{j-1}$ for all $j>i+1$. If
  $i=n$, then $b_i^n(h_0,\dots,h_n)=(g_0,\dots,g_{n-1})$ implies that
  $h_j=g_j$ for all $ j<n$.

  Let $m\in \ZZ[\Gg^{(n+1)}]$ and let $x\in \go$. Then
  \begin{multline*}
    \Phi_n\circ
    b_n(m)(x)=\sum_{p(g_0,\dots,g_{n-1})=x}b_n(m)(g_0,\dots,g_{n-1})[g_0,\dots,g_{n-1}]\\
    =\sum_{p(g_0,\dots,g_{n-1})=x}\sum_{i=0}^n(-1)^i
    b_{i*}^n(m)(g_0,\dots,g_{n-1})[g_0,\dots,g_{n-1}]\\
    =\sum_{p(g_0,\dots,g_{n-1})=x}\left( \sum_{i=0}^{n-1}(-1)^i
      \sum_{r(h_i)=r(g_i)}m(g_0,\dots,h_i,h_{i}^{-1}g_i,\dots,g_{n-1})[g_0,\dots,g_{n-1}]\right.\\
  \left.+(-1)^n\sum_{r(h_n)=s(g_{n-1})}m(g_0,\dots,g_{n-1},h_n)[g_0,\dots,g_{n-1}]
  \right)\\
  =\sum_{i=0}^{n-1}(-1)^i\sum_{p(g_0,\dots,g_{n-1})=x}\sum_{r(h_i)=r(g_i)}m(g_0,\dots,h_i,h_{i}^{-1}g_i,\dots,g_{n-1})[g_0,\dots,g_{n-1}]\\
  +(-1)^n\sum_{p(g_0,\dots,g_{n-1})=x}\sum_{r(h_n)=s(g_{n-1})}m(g_0,\dots,g_{n-1},h_n)[g_0,\dots,g_{n-1}].
\end{multline*}
Relabel $g_j$  as $h_j$ for $j<i$. If $i<n$ we label
$h_i^{-1}g_i$ as $h_{i+1}$ and note that $g_i=h_ih_{i+1}$.  Relabel
$g_j$ as $h_{j+1}$ for $j>i+1$. Then the above sums equal
\begin{multline*}
 =\sum_{i=0}^{n-1}(-1)^i \sum_{p(h_0,\dots,h_n)=x}m(h_0,\dots
 h_i,h_{i+1},\dots,h_n)[h_0,\dots,h_ih_{i+1},\dots,h_n]\\
 +(-1)^n\sum_{p(h_0,\dots,h_n)=x}m(h_0,\dots
 h_i,h_{i+1},\dots,h_n)[h_0,\dots,h_{n-1}]\\
 =\sum_{p(h_0,\dots,h_n)=x}m(h_0,\dots
 h_i,h_{i+1},\dots,h_n)\left( \sum_{i=0}^{n-1}(-1)^i
   [h_0,\dots,h_ih_{i+1},\dots,h_n]\right. \\
 +(-1)^n[h_0,\dots,h_{n-1}]\Biggr)=\od_n(\Phi_{n+1}(m)(x))=\ob_n\circ \Phi_{n+1}(m)(x).
\end{multline*}
 One can check separately that $\Phi_0\circ b_0=\ob_0\circ \Phi_1$.
\end{proof}
For $(g_1,\dots,g_n)\in \Gg^{(n)}$ and $V$ a compact open subset of
$\Gg^{(n)}$ with $(g_1,...,g_n)\in V$ such that $p\!\mid_V:V\to p(V)$ is a homeomorphism, we write
$\la g_1,\dots,g_n\ra_V$ for the element in $\ZZ[\Gg^{(n)}]$ defined
 in \eqref{eq:x_V} in a more general setting:
\begin{equation}
  \label{eq:section_U_2}
  \la g_1,\dots,g_n\ra_V(h_1,\dots, h_n)=
  \begin{cases}
    1 &\text{ if } (h_1,\dots,h_n)\in V\\
    0 & \text{ otherwise}.
  \end{cases}
\end{equation}
In particular, $\la g_1,\dots, g_n\ra_V(g_1,\dots,g_n)=1$.
Using \eqref{eq:y_V^s}, the corresponding  section $\langle
g_1,\dots,g_n \rangle^s_V$ of $\ZZ[\Gg^{(n)}]^s$   is given by
\begin{equation}
  \label{eq:section_U_s}
  \langle g_1,\dots,g_n \rangle^s_V(x)=
  \begin{cases}
    [h_1,\dots,h_n] & \text{ if }(h_1,\dots,h_n)\in V \text{ and
                      }r(h_1)=x\in p(V)\\
    0 & \text{otherwise}.
  \end{cases}
\end{equation}
 Hence $\langle g_1,\dots,g_n \rangle^s_V(r(g_1))=[g_1,\dots,g_n]$.

\medskip



 Let $\Gg$ be an ample groupoid, and let $M$ be a $\Gg$-module. To
 define  $n$-cochains with values in $M$, we use the equivalence of
 $\Gg$-modules and $\Gg$-sheaves, and we define the $n$-cochains to
 take values in the associated sheaf $\Mm$. Recall  that we can identify $M$ with
$\Gamma_c(\Mm,\pi)$, where $\Mm$ is the $\Gg$-sheaf constructed in Remark 
\ref{sheaf-mod}.  The set of $n$-cochains $C^n(\Gg,M)$ with
values in $M\cong 
\Gamma_c(\Mm,\pi)$  is 
\[C^n(\Gg,M)=\{f:\Gg^{(n)}\to
\Mm\;\mid\; f\,\text{ continuous, } f(g_1,\dots,g_n)\in M_{r(g_1)}\}.\] 
Then $C^n(\Gg,M)$ becomes an abelian group with pointwise
addition.  Note that $C^n(\Gg, M)$ can be identified with
$\Gamma(p^*\Mm,\pi_n)$, where $p^*\Mm$ is the pullback sheaf on
$\Gg^{(n)}$ with projection $\pi_n$.

The differentials are defined for $n\ge 1$ by
\[\delta_c^n:C^n(\Gg, M)\to C^{n+1}(\Gg,
  M),\;(\delta_c^nf)(g_0,g_1,...,g_n)=\]\[=g_0\cdot
  f(g_1,...,g_n)+\sum_{i=1}^n(-1)^if(g_0,...,g_{i-1}g_i,...,g_n)+(-1)^{n+1}f(g_0,...,g_{n-1}).\]

For $n=0$, let $(\delta_c^0f)(g_0)=g_0\cdot  f(s(g_0))-f(r(g_0))$

\begin{dfn}
The $M$-valued  cocycle cohomology is defined as \[H_c^n(\Gg, M)=(\ker \delta_c^n)/( \text{im } \delta_c^{n-1}),\] where $\delta^n_c$ are as above and $\delta_c^{-1}=0$. 
\end{dfn}
The particular case $M=\Gamma_c(\Mm, \pi)$ with $\Mm=\Gg^{(0)}\times A$ where $A$ is a topological abelian group and $g\cdot(s(g),a)=(r(g),a)$ gives $H_c^n(\Gg, A)$, the  cocycle cohomology with constant coefficients.

In the next theorem, it is important to consider the $\Gg$-module
$\ZZ[\Gg^{(n)}]$ in conjunction with  the corresponding $\Gg$-sheaf
$\ZZ[\Gg^{(n)}]^s$, see Remark \ref{sheaf-mod}.   First we recall that a morphism of
$\Gg$-sheaves $\Aa$ and $\Bb$  is a continuous map $f:\Aa\to \Bb$ such that
\begin{itemize}
\item for all $x\in \Gg^{(0)}$ and $a\in \Aa_x$ we have $f(a)\in \Bb_x$ and the induced map $\Aa_x\to \Bb_x$ is a homomorphism;
\item for any $(g,a)\in \Gg\ast\Aa, f(\alpha_g(a))=\beta_g(f(a))$, where $\beta$ is the action of $\Gg$ on $\Bb$.
\end{itemize} (see Definition 3.4 from
\cite{GK}).
We
will also use the following result, which is a particular case of
\cite[Proposition 3.3]{S14}; we prove
it here for completeness. 
\begin{prop}\label{lem:hom_isom}
  Assume that $Y$ is an \'etale $\Gg$-space with map $p:Y\to \go$. 
  There is an isomorphism \[\Xi:\text{Hom}_{\Gg}(\Gamma_c(\ZZ[Y]^s,p^s),
  \Gamma_c(\Mm,\pi))\to \text{Hom}_\Gg(\ZZ[Y]^s,
  \Mm)\] defined via
  \begin{equation}\label{eq:Xi}
    \Xi(\varphi)([y])=\varphi(\langle y  \rangle^s_V)(p(y)),
  \end{equation}
  for $\varphi\in \text{Hom}_{\Gg}(\Gamma_c(\ZZ[Y]^s,p^s),
  \Gamma_c(\Mm,\pi))$ and  $[y]\in \ZZ[Y]^s$, where $V$ is a
  compact open neighborhood of $y$ such that $p|_V$ 
  is a homeomorphism onto its image, and
  the section  $\langle y \rangle^s_V$ was defined in
  \eqref{eq:y_V^s}. Its  inverse is defined via
   \[
     \Xi^{-1}(f)(\xi)(x)=f(\xi(x))
   \]
   for all $f\in \text{Hom}_\Gg(\ZZ[Y]^s, \Mm)$, $\xi\in
   \Gamma_c(\ZZ[Y]^s,p^s)$ and $x\in \go$.
In particular,
\[\text{Hom}_{\Gg}(\ZZ[Y], M)\cong
  \text{Hom}_\Gg(\ZZ[Y]^s, \Mm)\]
and the map $\Xi$ is natural with respect to morphisms of \'etale $\Gg$-spaces.
\end{prop}
\begin{proof}
  We use Remark \ref{sheaf-mod} to identify $M$ with $\Gamma_c(\Mm,\pi)$ and Lemma \ref{lem:isom_gamma_n} to
  identify $\ZZ[Y]$ with $\Gamma_c(\ZZ[Y]^s,p^s)$. Hence we can identify  $\text{Hom}_{\Gg}(\ZZ[Y], M)$ with
  $\text{Hom}_{\Gg}(\Gamma_c(\ZZ[Y]^s,p^s),
  \Gamma_c(\Mm,\pi))$.

  We note that the definition of $\Xi$ is independent of
  the choice of the compact open neighborhood $V$. Indeed, assume that
  $W$ is another compact open neighborhood of $y$ such
  that $p|_W$ is a homeomorphism. Let $U:=p(V\bigcap W)$. Then $U$
  is a compact open subset of $\go$ and
  \begin{multline*}
    \varphi(\langle y  \rangle^s_V)(p(y))=(\chi_U\cdot
    \varphi(\langle y \rangle^s_V))(p(y))
    =\varphi(\chi_U\cdot \langle y \rangle^s_V)(p(y))=\varphi(\langle y \rangle^s_{UV})(p(y)) \\
    =\varphi(\langle y \rangle^s_{UW})(p(y))=\varphi(\chi_U\cdot \langle  y
    \rangle^s_W)(p(y))
    =(\chi_U\cdot \varphi(\langle y \rangle^s_W))(p(y))\\
    =\varphi(\langle  y \rangle^s_W)(p(y)).
  \end{multline*}
   By definition,  $\Xi(\varphi)([y])\in M_{p(y)}$. Hence the first
  condition of a sheaf homomorphism is satisfied. To check the second
  condition, let $g\in \Gg$ and $y\in Y$. Let $V$ be
  a compact open neighborhood of $y$ such that $p|_V$
  is a homeomorphism onto its image, and
  let $U\in \Bis(\Gg)$ such that $g\in U$. Then
  \begin{multline*}
     \Xi(\varphi)(g\cdot [y])= \Xi(\varphi)([g\cdot
     y])=\varphi(\langle  gy
     \rangle^s_{UV})(r(g))
     =\varphi(\chi_U\cdot\langle y \rangle^s_V)(r(g))\\=(\chi_U\cdot
     \varphi(\langle y \rangle^s_V))(r(g))
     =g\cdot (\varphi(\langle y \rangle^s_V)(p(y))).
   \end{multline*}
   We check that
   $\Xi^{-1}(f)$ is a $\ZZ[\Gg]$-homomorphism, for all $f\in
   \text{Hom}_\Gg(\ZZ[Y]^s, \Mm)$. Let $a\in \ZZ[\Gg]$ and
   $\xi\in \Gamma_c(\ZZ[Y]^s,p^s)$. We have
   \begin{multline*}
     \Xi^{-1}(f)(a\cdot \xi)(x)=f((a\cdot \xi)(x))=f\left(
       \sum_{r(g)=x}a(g)g\cdot \xi(s(g)) \right)\\
     =\sum_{r(g)=x}a(g)g\cdot f(\xi(s(g)))=(a\cdot
     \Xi^{-1}(f)(\xi))(x).  
   \end{multline*}
   We prove next that $\Xi\circ \Xi^{-1}(f)=f$ and
   $\Xi^{-1}\circ \Xi(\varphi)=\varphi$.

    Under our assumption
  that $\Gg$ and  $Y$ are Hausdorff, it suffices
  to prove that $\Xi^{-1}\circ \Xi(\varphi)(\xi)=\varphi(\xi)$ for
  $\xi=\langle  y \rangle^s_V$ for all
  $[y]\in \ZZ[Y]^s$ and $V$ any compact open
  neighborhood of $y$ such that $p|_V$ is a
  homeomorphism onto its image.  We have
  \begin{multline*}
    \bigl(  \Xi^{-1}\circ \Xi(\varphi)\bigr)(\langle y
    \rangle^s_V)(x)=\Xi(\varphi)(\langle y \rangle^s_V(x))\\
    =
    \begin{cases}
      \Xi(\varphi)([z]) & \text{ if }z\in V\,\text{
                                        and }p(z)=x\in p(V)\\
      0 & \text{otherwise }
    \end{cases}\\
    =\begin{cases}
      \varphi(\langle z \rangle^s_V)(x) & \text{ if }z\in V\,\text{
                                                  and
                                                      }p(z)=x\in p(V)\\
      0 & \text{otherwise }
     \end{cases}\\
  =\varphi(\langle  y \rangle^s_V)(x),
   \end{multline*}
   where we used the fact that for a fixed $V$, $\langle z
   \rangle^s_V=\langle y \rangle^s_V$
   for all $z$ and $y$ in $V$.
   
   Let $f\in \text{Hom}_\Gg(\ZZ[Y]^s, \Mm)$ and
   $[y]\in \ZZ[Y]^s$. Then
   \begin{multline*}
     \bigl( \Xi\circ \Xi^{-1}(f) \bigr)([y])=\bigl(
     \Xi^{-1}f\bigr)(\langle y \rangle^s_V)(p(y))
     =f(\langle  y \rangle^s_V(p(y)))=f([y]),
   \end{multline*}
   where $V$ is a compact open neighborhood of $y$ such
   that $p|_V$ is a homeomorphism onto its image. 
\end{proof}

\begin{rmk}\label{rem:Xi_n} If $Y=\Gg^{(n)}$, we write $\Xi_n$ for the corresponding isomorphism
  \[\Xi_n:\text{Hom}_{\Gg}(\Gamma_c(\ZZ[\Gg^{(n)}]^s,p^s),
  \Gamma_c(\Mm,\pi))\to \text{Hom}_\Gg(\ZZ[\Gg^{(n)}]^s,
  \Mm)\] defined via
  \begin{equation}\label{eq:Xi_n}
    \Xi_n(\varphi)([g_1,\dots,g_n])=\varphi(\langle g_1,\dots,g_n  \rangle^s_V)(r(g_1)).
  \end{equation}
  Therefore \[\text{Hom}_{\Gg}(\ZZ[\Gg^{(n)}], M)\cong
  \text{Hom}_\Gg(\ZZ[\Gg^{(n)}]^s, \Mm)\]
and the map $\Xi_n$ is natural with respect to morphisms of \'etale $\Gg$-spaces.
\end{rmk}

If $\varphi\in \text{Hom}_{\Gg}(\ZZ[\Gg^{(n)}],M)$ then $\eta_M\circ
\varphi\in \text{Hom}_{\Gg}(\ZZ[\Gg^{(n)}],\Gamma_c(\Mm,\pi))$, where
$\eta_M$ is the isomorphism defined in Remark \ref{sheaf-mod}. To keep the
notation cleaner we do not write $\eta_M$ in the remaining of the
paper. That is, we identify $\text{Hom}_{\Gg}(\ZZ[\Gg^{(n)}],M)$ with
$\text{Hom}_{\Gg}(\ZZ[\Gg^{(n)}],\Gamma_c(\Mm,\pi))$ via composition
with $\eta_M$. The following result combines the equivalence between
$\Gg$-modules and $\Gg$-sheaves for ample groupoids, \cite[Theorem 3.5]{S14}, with
\cite[Proposition 3.14]{GK}.

\begin{thm}\label{thm:main}
Let $\Gg$ be an ample groupoid and let $M$ be a $\Gg$-module. For each $n\ge 0$ there is an isomorphism $\theta^n:\text{Hom}_{\Gg}(\ZZ[\Gg^{(n+1)}], M)\to C^n(\Gg, M)$ determined by
\[(\theta^n\varphi)(g_1,...,g_n)=\varphi(\la r(g_1),g_1,...,g_n\ra_V )(r(g_1)),\]
for all $\varphi\in \text{Hom}_{\Gg}(\ZZ[\Gg^{(n+1)}], M)$ and
$(g_1,\dots,g_n)\in \Gg^{(n)}$, where for
$V$  a  compact open subset of $\Gg^{(n+1)}$ such that
$(r(g_1),g_1,\dots,g_n)\in V$ and 
$p|_V$ is a homeomorphism, $\la
r(g_1),g_1,...,g_n\ra_V$ is the function defined in equation \eqref{eq:section_U_2}.

The map $\theta^n$  is compatible with the boundary maps, and  induces
an isomorphism $H^n(\Gg, M)\cong H^n_c(\Gg, M)$. The inverse is
induced by $\rho^n:C^n(\Gg, M)\to \text{Hom}_{\Gg}(\ZZ[\Gg^{(n+1)}],
M)$ determined by  
\[(\rho^nf)(\la g_0,g_1,...,g_n\ra_W)(x)=
  \begin{cases}
    h_0\cdot f(h_1,\dots,h_n) & \text{ if } (h_0,\dots, h_n)\in W\\
                             &   \text{ and } p(h_0,\dots,h_n)=x\\
    0 & \text{ otherwise.}
  \end{cases}
\] for $f\in C^n(\Gg, M)$ and $W$ a compact open subset of
$\Gg^{(n+1)}$ such that $p|_W$ is a homeomorphism and
$(g_0,\dots,g_n)\in W$.
\end{thm}
  
\begin{proof}

We mention that the map $\theta^n$ is the composition of the map $\xi^n$ defined in
  \cite[Proposition 3.14]{GK} with the map $\Xi_n$ defined in
  \eqref{eq:Xi_n} and the map $\Phi_n$ defined in Remark \ref{rm:Phi_n}. The map $\xi^n$ is given via
  \[
    (\xi^nf)(g_1,\dots,g_n)=f([r(g_1),g_1,\dots,g_n])
  \]
  with inverse $\eta^n$ defined via
  \[
    (\eta^nf)([g_0,g_1,\dots,g_n])=g_0\cdot f(g_1,\dots,g_n).
  \]

Using  the proof of Proposition \ref{lem:hom_isom}, the
definition of $(\theta^n\varphi)(g_1,\dots,g_n)$ is independent of the
compact open set $V$. For each $\varphi\in
\text{Hom}_{\Gg}(\ZZ[\Gg^{(n+1)}], M)$, $\theta^n\varphi$ is
continuous, since $\varphi(f)$ is a continuous section for any $f\in
\ZZ[\Gg^{n+1}]$ and $r$ is a local homeomorphism.  

A routine computation shows that $\delta^n_c(\theta^n\varphi)=\theta^{n+1}(\delta_n(\varphi))$, in other words, $\theta^n$ 
takes cocycles to cocycles and coboundaries to coboundaries, so it  induces a homomorphism $H^n(\Gg, M)\to H^n_c(\Gg, M)$. Indeed, we have 
\[\delta^n_c(\theta^n\varphi)(g_0,g_1,...,g_n)
=g_0\cdot(\theta^n\varphi)(g_1,...,g_n)+\]\[+\sum_{i=1}^n(-1)^i(\theta^n\varphi)(g_0,...,g_{i-1}g_i,...,g_n)+(-1)^{n+1}(\theta^n\varphi)(g_0,...,g_{n-1})\]\[=g_0\cdot \varphi(\la r(g_1), g_1,...,g_n\ra_V)(r(g_1))+\]\[+\sum_{i=1}^n(-1)^i\varphi(\la r(g_0),g_0,...,g_{i-1}g_i,...,g_n\ra_{V_i})(r(g_0))+\]\[+(-1)^{n+1}\varphi(\la r(g_0),g_0,...,g_{n-1}\ra_{V_{n+1}})(r(g_0))\]
and
\[\theta^{n+1}(\varphi\circ b_{n+1})(g_0,g_1,...,g_n)=(\varphi\circ b_{n+1})(\la r(g_0),g_0,g_1,...,g_n\ra_U)(r(g_0))=\]\[=\varphi(\sum_{i=0}^{n+1}(-1)^ib^{n+1}_{i*}(\la r(g_0),g_0,g_1,...,g_n\ra_U))(r(g_0))=\]\[=\varphi(\la g_0,g_1,...,g_n\ra_W)(r(g_0))+\]\[+\sum_{i=1}^n(-1)^i\varphi(\la r(g_0),g_0,...,g_{i-1}g_i,...,g_n\ra_{V_i})(r(g_0))+\]\[+(-1)^{n+1}\varphi(\la r(g_0),g_0,...,g_{n-1}\ra_{V_{n+1}})(r(g_0)).\]
The equality holds since 
\[g_0\cdot \varphi(\la r(g_1),g_1,...,g_n\ra_V )(r(g_1))=\varphi(\la g_0,g_1,...,g_n\ra_W)(r(g_0)),\]

\bigskip
The fact that
$\delta^n_c(\theta^n\varphi)=\theta^{n+1}(\delta_n(\varphi))$ also follows 
from Proposition \ref{prop:sheaf_coh}, Proposition \ref{lem:hom_isom}
and \cite[Proposition 3.14]{GK} and so does the fact that $\theta^n$
is invertible with inverse $\rho^n$. 

\end{proof}

\begin{cor}
Given an ample groupoid $\Gg$, the cocycle cohomology $H^*_c(\Gg,\Mm)$ coincides with $H^*(\Gg,M)$, where $M=\Gamma_c(\Mm, \pi)$. Using section 8 in \cite{T06}, it follows that equivalent groupoids have the same cohomology. If the sheaf $\Mm$ is the trivial sheaf $\underline{\ZZ}=\Gg^{(0)}\times \ZZ$, then we write $H^n(\Gg,\ZZ)$ for the cohomology groups with constant coefficients $\ZZ$. 
\end{cor}

As another consequence of Steinberg's equivalence theorem
(\cite[Theorem 3.5]{S14}) as applied
in Proposition \ref{lem:hom_isom}, we can describe the dependence of $H^*$
on $\Gg$. We sketch the details next.

Recall (see, for example, \cite[\S 0]{K88}), that if $\phi:\Gg_1\to
\Gg_2$ is an \'etale groupoid homomorphism then one can define the
pullback functor $\phi^*$ from the category $\Ss(\Gg_2)$ of 
$\Gg_2$-sheaves and $\Gg_2$-morphisms of sheaves, to the category
$\Ss(\Gg_1)$ as follows. If  $\Aa$ is a
$\Gg_2$-sheaf, then the pullback $\Gg_1$-sheaf is
\[
  \phi^*\Aa=\{ (x,a)\,:\,x\in \Gg_1^{(0)}, a\in \Aa_{\phi(x)}\}.
\]
The action of $\Gg_1$ is defined via $g\cdot
(s(g),a):=(r(g),\phi(g)\cdot a)$.
If $f:\Aa\to \Bb$ is a morphism of $\Gg_2$-sheaves, then
the pullback morphism $\phi^*(f):\phi^*\Aa\to \phi^*\Bb$ is
defined via $\phi^*(f)(x,a)=(x,f(a))$. 
We define $\psi_n:\ZZ[\Gg_1^{(n)}]^s\to
\phi^*(\ZZ[\Gg_2^{(n)}]^s)$ via 
\[
  \psi_n(\sum a_{(g_1,\dots,g_n)}[g_1,\dots,g_n])=(x,\sum a_{(g_1,\dots,g_n)}[\phi^{(n)}(g_1,\dots,g_n)]),
\]
where $x=p^s(\sum a_{(g_1,\dots,g_n)}[g_1,\dots,g_n])$.
Hence, if
$f\in \text{Hom}_{\Gg_2}(\ZZ[\Gg_2^{(n)}]^s,\Mm)$, where $\Mm$ is a
$\Gg_2$-sheaf, then $\phi^*(f)\circ \psi_n\in
\text{Hom}_{\Gg_1}(\ZZ[\Gg_1^{(n)}]^s,\phi^*\Mm)$. We write 
$\phi^*(f)$ instead of $\phi^*(f)\circ \psi_n$ in the remainder of the
paper to slightly simplify  the notation.
 
\begin{cor}\label{cor:pullback} Consider an \'etale groupoid
homomorphism $\phi:\Gg_1\to \Gg_2$ between ample 
groupoids. If $M$ is a $\Gg_2$-module, then we identify $M$ with
$\Gamma_c(\Mm,\pi)$ and we define the pullback $\Gg_1$-module
$\phi^*M:=\Gamma_c(\phi^*(\Mm),\pi)$.  The map $\phi$ induces  homomorphisms
\[\hat{\phi}^{(n)}:\text{Hom}_{\Gg_2}(\ZZ[\Gg_2^{(n)}], M)\to
  \text{Hom}_{\Gg_1}(\ZZ[\Gg_1^{(n)}], \phi^*M), \]\[
  \hat{\phi}^{(n)}(h)=\Xi_n^{-1}(\phi^*(\Xi_n(h))),\]  
for all $h\in \text{Hom}_{\Gg_2}(\ZZ[\Gg_2^{(n)}], M)$,
where  $\Xi_n$ was defined in  
Remark \ref{rem:Xi_n}. Also, since $\hat{\phi}^{(n)}$ are compatible
with the coboundary maps, $\phi$ determines  cohomology group
homomorphisms \[\phi^*:H^*(\Gg_2,M)\to 
H^*(\Gg_1,\phi^*M)\] 
and $\phi\mapsto \phi^*$ reverses composition.
\end{cor}

\begin{rmk}\label{rm:explicit_phi_star}
 When we
  identify $M$ with $\Gamma_c(\Mm,\pi)$ and $\phi^*M$ with
  $\Gamma_c(\phi^*(\Mm),\pi)$, the homomorphism $\hat{\phi}^{(n)}$ has the following explicit formula.  
  
  \noindent Let $h\in
  \text{Hom}_{\Gg_2}(\ZZ[\Gg_2^{(n)}],\Gamma_c(\Mm,\pi))$ and $f\in
  \ZZ[\Gg_1^{(n)}]$. Then
  \[
    \hat{\phi}^{(n)}(h)(f)(x)=\bigl(x,\sum_{r(g_1)=x}f(g_1,\dots,g_n)h(\langle \phi^{(n)}(g_1,\dots,g_n) \rangle_{V_{g_1,\dots,g_n}})(\phi(x))\bigr),
  \]
 where $V_{g_1,\dots,g_n}$ are compact open subsets of $\Gg_2^{(n)}$ such that
 $\phi^{(n)}(g_1,\dots,g_n)\in V_{g_1,\dots,g_n}$ and the
 restriction of the anchor map $p$ to each of these sets is a homeomorphism. 
\end{rmk}

\begin{rmk}\label{rm:injective-surjective}
   In the next section, we will use the fact that if $\phi$ is a
  surjective \'etale groupoid homomorphism then $\hat{\phi}^{(n)}$ and  $\phi^*$ are
  injective. Indeed, assume that $h\in
  \text{Hom}_{\Gg_2}(\ZZ[\Gg_2^{(n)}], M)$ is such that
  $\hat{\phi}^{(n)}(h)=0$. Since $\Xi_n^{-1}$ is an isomorphism, it
  follows that $\phi^*(\Xi_n(h))=0$. Let $(g_1,\dots,g_n)\in
  \Gg_1^{(n)}$ and $U$ a compact open neighborhood of
  $(g_1,\dots,g_n)$  such
  that $p|_U$ is a homeomorphism onto $p(U)$. Then
  \[
    \phi^*(\Xi_n(h))(\langle g_1,\dots, g_n \rangle_U)=0.
  \]
  By the definition of $\phi^*$ or, more precisely, $\phi^*(\cdot
  )\circ \psi_n$,
  \[
    (\phi(r(g_1)),\Xi_n(h)([\phi(g_1),\dots,\phi(g_n)]))=0,
  \]
  which, by the definition of $\Xi_n$, implies that $h(\langle
  \phi(g_1),\dots,\phi(g_n) \rangle_V)=0$, where $V$ is any compact open
  neighborhood of $(\phi(g_1),\dots,\phi(g_n))$ in $\Gg_2^{(n)}$ such that
  $p|_V$ is a homeomorphism onto $p(V)$. Since $\phi$ is surjective,
  the span of the set of functions $\langle   \phi(g_1),\dots,\phi(g_n) \rangle_V$
  is dense in $\ZZ[\Gg_2^{(n)}]$. Hence $h=0$ and, thus, $\hat{\phi}^{(n)}$ is injective.
\end{rmk}
\bigskip

\section{The exact sequence of cohomology for a cocycle}
\bigskip
We recall the definition of the skew product groupoid and prove an exact sequence of cohomology, our second main result.

Let $\Gg$ be an \'etale groupoid. If $c:\Gg\to \ZZ$ is a continuous homomorphism, the skew product groupoid $\Gg\times_c\ZZ$ has unit space identified with $\Gg^{(0)}\times \ZZ$ and for $(g,k)\in \Gg\times \ZZ$,
\[r(g,k)=(r(g),k),\;\; s(g,k)=(s(g), k+c(g)),\]with multiplication and inverse\[(g,k)(h, k+c(g))=(gh, k),\;\; (g,k)^{-1}=(g^{-1}, k+c(g)).\]
There is an action $\hat{c}: \ZZ\curvearrowright \Gg\times_{c}\ZZ$ with generator $\hat{c}_1(g,k)=(g, k+1)$. Note that $\hat{c}_1:\Gg\times_c\ZZ\to \Gg\times_c\ZZ$ is a groupoid isomorphism.

To compute the homology of certain Exel-Pardo groupoids associated to self-similar actions without using spectral sequences, Ortega proved in Lemma 1.3 of \cite{O20} the existence of a long exact sequence of homology; see also section 3.2.1 in \cite{S} for a simplified proof. More precisely, for an ample groupoid $\Gg$  and a cocycle $c:\Gg\to \ZZ$, there is an exact sequence in homology with coefficients in $\ZZ$
 \bigskip
\[0\longleftarrow H_0(\Gg)\longleftarrow H_0(\Gg\times_c\ZZ)\stackrel{id-c_*^{(0)}}{\longleftarrow}H_0(\Gg\times_c\ZZ)\longleftarrow H_1(\Gg)\longleftarrow\cdots\]
\[\longleftarrow H_n(\Gg)\longleftarrow H_n(\Gg\times_c\ZZ)\stackrel{id-c_*^{(n)}}{\longleftarrow}H_n(\Gg\times_c\ZZ)\longleftarrow H_{n+1}(\Gg)\longleftarrow\cdots\]
Here $c_*^{(n)}:  \ZZ[(\Gg\times_c\ZZ)^{(n)}]\to \ZZ[(\Gg\times_c\ZZ)^{(n)}]$ are the maps induced by the generator $\hat{c}_1$ of the action $ \ZZ\curvearrowright \Gg\times_{c}\ZZ$ and we also denote by $c_*^{(n)}$ the induced maps between homology
groups.  Note that
\[c_*^{(0)}:\ZZ[\Gg\times_c\ZZ]\to \ZZ[\Gg\times_c\ZZ],\; c_*^{(0)}(f)(g,k)=f(g,k-1)\]
and that $c_*^{(n)}$ are $\Gg\times_c\ZZ$-module maps. We will prove that there is a dual long exact sequence for cohomology.

\begin{rmk}\label{pi}
 
The map
  $\pi:\Gg\times_c\ZZ\to \Gg$, $\pi(g,k)=g$ is an onto \'etale groupoid
  homomorphism. Therefore, if $M$ is a $\Gg$-module, we can apply
  Corollary \ref{cor:pullback} and obtain the pullback
  $\Gg\times_c\ZZ$-module $\pi^*M$ and a homomorphism
  \[\hat{\pi}^{(n)}:\text{Hom}_{\Gg}(\ZZ[\Gg^{(n)}],M)\to 
  \text{Hom}_{\Gg\times_c \ZZ}(\ZZ[(\Gg\times_c\ZZ)^{(n)}],\pi^*M)\] compatible with the coboundary maps.
  
  Also, the groupoid isomorphism $\hat{c}_1:\Gg\times_c\ZZ\to \Gg\times_c\ZZ$ determines an isomorphism
  \[\hat{c}^{(n)}:\text{Hom}_{\Gg\times_c\ZZ}(\ZZ[(\Gg\times_c\ZZ)^{(n)}], \pi^*M)\to \text{Hom}_{\Gg\times_c\ZZ}(\ZZ[(\Gg\times_c\ZZ)^{(n)}],\pi^*M)\]
  since $\hat{c}_1^*\pi^*M= \pi^*M$.
 
\end{rmk}

\begin{thm}\label{les}
Given $\Gg$ an ample groupoid  and a cocycle $c:\Gg\to\ZZ$, for any $\Gg$-module $M$ we have a long exact sequence in cohomology
\[0\rightarrow H^0(\Gg, M)\rightarrow H^0(\Gg\times_c \ZZ,\pi^*M)\stackrel{id-c^{*(0)}}{\longrightarrow}H^0(\Gg\times_c \ZZ, \pi^*M)\rightarrow H^1(\Gg, M)\rightarrow\cdots\]
\[\rightarrow H^n(\Gg, M)\rightarrow H^n(\Gg\times_c \ZZ, \pi^*M)\stackrel{id-c^{*(n)}}{\longrightarrow}H^n(\Gg\times_c \ZZ, \pi^*M)\rightarrow H^{n+1}(\Gg, M)\rightarrow\cdots,\]
where $\pi^*M$ is the pullback $\Gg\times_c\ZZ$-module and we  denote by $c^{*(n)}$ the induced maps between cohomology groups.
\end{thm}

\begin{proof}
We claim that for each $n$ we have a short exact sequence 
\[
0\to \text{Hom}_{\Gg}(\ZZ[\Gg^{(n)}],M)\stackrel{\hat{\pi}^{(n)}}{\longrightarrow}\text{Hom}_{\Gg\times_c\ZZ}(\ZZ[(\Gg\times_c\ZZ)^{(n)}],\pi^*M)\stackrel{id-\hat{c}^{(n)}}{\longrightarrow}\]\[\stackrel{id-\hat{c}^{(n)}}{\longrightarrow}\text{Hom}_{\Gg\times_c\ZZ}(\ZZ[(\Gg\times_c\ZZ)^{(n)}],\pi^*M)\to 0,
\]
where $\hat{\pi}^{(n)}$ and $\hat{c}^{(n)}$ were defined in Remark \ref{pi}. 

Indeed,  $\hat{\pi}^{(n)}$ is injective  since $\pi^{(n)}:(\Gg\times_c\ZZ)^{(n)}\to \Gg^{(n)}$ is onto.
(see Remark \ref{rm:injective-surjective}).

Since $\pi\circ id=\pi\circ \hat{c}_1$ as groupoid homomorphisms $\Gg\times_c\ZZ\to\Gg\times_c\ZZ\to \Gg$, we obtain $(id-\hat{c}^{(n)})\circ \hat{\pi}^{(n)}=0$ and hence $\text{im } \hat{\pi}^{(n)}\subseteq \ker(id-\hat{c}^{(n)})$. Since $\hat{c}_1$ does not have fixed points, it follows that $id-\hat{c}^{(n)}$ is onto.
We only need to prove that $\ker(id-\hat{c}^{(n)})\subseteq \text{im }\hat{\pi}^{(n)}$. 

Let  $\lambda\in \text{Hom}_{\Gg\times_c\ZZ}(\ZZ[(\Gg\times_c\ZZ)^{(n)}],\pi^*M)$ such that $\hat{c}^{(n)}(\lambda)=\lambda$. We need to find $\varphi \in \text{Hom}_{\Gg}(\ZZ[\Gg^{(n)}],M)$ such that $\hat{\pi}^{(n)}(\varphi)=\lambda$. 

For the elements of $(\Gg\times_c\ZZ)^{(n)}$ we use the notation $((g_1,k_1),\dots ,(g_n,k_n))$ instead of $((g_1,k),(g_2,k+c(g_1)),\dots ,(g_n,k+c(g_1)+\cdots +c(g_{n-1}))$. Recall that 
\[\hat{c}_1^{(n)}:(\Gg\times_c\ZZ)^{(n)}\to (\Gg\times_c\ZZ)^{(n)},\]\[ \hat{c}_1^{(n)}((g_1,k_1),\dots ,(g_n, k_n))=((g_1, k_1+1),\dots,(g_n,k_n+1)).\]

In the next argument, the multiple use of $p$ as the anchor map from
   $\Gg^{(n)}$ onto $\Gg^{(0)}$ and as the anchor map from
   $(\Gg\times_c\ZZ)^{(n)}$ onto $\Gg^{(0)}\times \ZZ$ should be clear
 from the context. For $((g_1,k_1),\dots,(g_n,k_n))\in(\Gg\times_c\ZZ)^{(n)}$ and $l\in \ZZ$, consider 
 $V$  a compact open neighborhood of $((g_1,k_1),\dots,(g_n,k_n))$
 such that $p|_V$ is a homeomorphism, and $V_l$  a compact open
 neighborhood of $((g_1,k_1+l),\dots,(g_n,k_n+l))$ such that
 $p|_{V_l}$ is a homeomorphism. Using the explicit formula from Remark \ref{rm:explicit_phi_star}, the  fact that $\hat{c}^{(n)}(\lambda)=\lambda$ implies that  \[
   \lambda(\langle ((g_1,k_1),\dots,(g_n,k_n) )\rangle_V)(x,k)=
   \lambda(\langle ((g_1,k_1+l),\dots,(g_n,k_n+l)) \rangle_{V_l})(x,k+l).
 \]
 
Consider  $\varphi \in \text{Hom}_{\Gg}(\ZZ[\Gg^{(n)}],M)$ defined via
 \[
   \varphi(\langle g_1,\dots,g_n \rangle_U)(x):=
  \lambda(\langle  ((g_1,k_1),\dots,(g_n,k_n)) \rangle_V)(x,k),
 \]
 where   $U$ is a
 compact open neighborhood of $(g_1,\dots,g_n)\in \Gg^{(n)}$ such that $p|_U$
 is a homeomorphism. The map $\varphi$ is well defined, since  if $(x,k+l)$ is another element in
 $\pi^{-1}(x)$, then the only element in 
 $V_l\bigcap (\pi^{(n)})^{-1}(g_1,\dots,g_n)$ such
 that $(x,k+l)\in p(V_l)$ is $((g_1,k_1+l),\dots,(g_n,k_n+l))$. 
Using again Remark \ref{rm:explicit_phi_star}, it follows that $\hat{\pi}^{(n)}(\varphi)=\lambda$ and hence $\ker(id-\hat{c}^{(n)})= \text{im }\hat{\pi}^{(n)}$.

Since the maps in the above short exact sequence are compatible with
the coboundary maps, we get a short exact sequence of cochain
complexes and  we can use   the associated long exact sequence of
cohomology  to get our result, see Theorem 1.3.1 in \cite{W}. 
\end{proof}

\begin{cor}
If we have a minimal homeomorphism of the Cantor set $X$, the cohomology of the action groupoid $\ZZ\ltimes X$ can  be computed using the above  long exact sequence, see also Example \ref{transf}.
\end{cor}
 
 \section{Examples}
 
 \bigskip
 We illustrate the theory by several computations of the cohomology groups.
 \begin{example}\label{ex:Coh_space}
Let $X$ be a zero-dimensional space (i.e. totally disconnected). For
$\Gg=X$ viewed as an ample groupoid with trivial multiplication, we
identify $\Gg^{(n)}$ with $X$ for all $n\ge 0$ and all the face maps
 $\partial_i^n:\Gg^{(n)}\to\Gg^{(n-1)}$ and
$b_i^n:\Gg^{(n+1)}\to
\Gg^{(n)}$   become the identity. 
Therefore, for $A$ an abelian group, the  differentials $d_n : C_c(X, A) \to C_c(X, A)$ are the zero maps for $n=0$ or $n$ odd and the identity for $n\ge 2$ even. It follows that \[H_0(X,A)=\ker d_0=C_c(X,A)\;\text{ and }\;H_n(X,A)=0\;\text{ for } n\ge 1.\]

If we dualize the chain complex for $M$ a $\Gg$-module, the differentials
$\delta_n:\text{Hom}_X(\ZZ[X],M)\to \text{Hom}_X(\ZZ[X],M)$ are the zero
maps if $n$ is even and the identity for $n$ odd since $b_n =id$ for $n$ even and $b_n = 0$ for n odd. 
If $M=\Gamma_c(X\times \ZZ)\cong \ZZ[X]$, we get the cochain complex  
\[\text{Hom}_X(\ZZ[X],\ZZ[X])\stackrel{\delta_0}{\rightarrow}
  \text{Hom}_X(\ZZ[X],\ZZ[X])\stackrel{\delta_1}{\rightarrow}
  \text{Hom}_X(\ZZ[X],\ZZ[X])\rightarrow\cdots\]
where $\delta_n=0$ for $n$ even and $\delta_n=id$ for $n$ odd. It follows that \[H^0(X,\ZZ)=\ker \delta_0\cong C(X,\ZZ) \text{ and
} H^n(X,\ZZ)=0 \text{ for } n\ge 1.\]
Indeed, $\ker \delta_0=\text{Hom}_X(\ZZ[X],\ZZ[X])$ and one can
identify $\text{Hom}_X(\ZZ[X],\ZZ[X])$ with
$\Gamma(X,\underline{\ZZ})\cong C(X,\ZZ)$, where $\underline{\ZZ}$ is the constant
sheaf over $X$ with fiber $\ZZ$, via the map that sends $\varphi\in
\text{Hom}_X(\ZZ[X],\ZZ[X])$ to the section defined by $x\mapsto
\varphi(\langle x \rangle_U)(x)$, where $U$ is any compact open 
neighborhood of $x\in X$.
\end{example}

\begin{rmk}

In Addendum 3 of \cite{K88}, Kumjian proves the existence of an exact sequence of sheaf cohomology for inductive limits of ultraliminary groupoids, involving the derived functor $\varprojlim{\!}^1$ of the projective limit functor $\varprojlim$, see also Example 4.3 in \cite{DKM}. Recall that for a sequence of abelian groups and homomorphisms \[\cdots\to A_2\stackrel{\alpha_2}{\rightarrow} A_1\stackrel{\alpha_1}{\rightarrow} A_0\] we define
\[\beta:\prod_{i=0}^\infty A_i\to \prod_{i=0}^\infty A_i, \; \beta((g_i))=(g_i-\alpha_{i+1}(g_{i+1}))\] and then $\varprojlim A_i=\ker \beta$ and $\varprojlim{\!}^1 A_i=$ coker $\beta$. 

More precisely, given a sequence of local homeomorphisms \[X_0\stackrel{\varphi_0}{\longrightarrow} X_1\stackrel{\varphi_1}{\longrightarrow} X_2\stackrel{\varphi_2}{\longrightarrow}\cdots\] with $X_n$ locally compact spaces, let 
\[\Gg_n=R(\psi_n)=\{(x,y)\in X_0\times X_0\;\mid\; \psi_n(x)=\psi_n(y)\}\] 
be the equivalence relation  on $X_0$ determined by \[\psi_n=\varphi_{n-1}\circ\cdots\circ\varphi_0:X_0\to X_n\] for $n\ge 1$, and let $\ds \Gg=\bigcup_{n=1}^\infty\Gg_n$. Then $\Gg_n$ has the same cohomology as $X_n$ and for all $q\ge 1$  there is a short exact sequence
\begin{equation}\label{eq:SESEQ_Alex}
  0\to \varprojlim{}^1H^{q-1}(X_n, \Aa^n)\to H^q(\Gg, \Aa)\to
  \varprojlim H^q(X_n, \Aa^n)\to 0,
\end{equation}
where $\Aa$ is a $\Gg$-sheaf and $\Aa^n$ is the sheaf over $X_n$ corresponding to $\Aa$. For $q=0$ it follows that $H^0(\Gg,\Aa)\cong \varprojlim H^0(X_n, \Aa^n)$.  


\end{rmk}

Recall from \cite{FKPS} that an ample groupoid $\Gg$ is called elementary if it is isomorphic to the equivalence relation \[R(\psi)=\{(y_1,y_2)\in Y\times Y\; \mid\; \psi(y_1)=\psi(y_2)\},\] determined by a local homeomorphism $\psi:Y\to X$ between zero-dimensional spaces. Since $X$ and $R(\psi)$ are equivalent groupoids via $Y$, they have the same homology. An ample groupoid $\Gg$ is called $AF$ if it is a union of open elementary subgroupoids with the same unit space.  If $\Gg$ is an $AF$-groupoid with unit space $X$, then $\Gg=\varinjlim \Gg_n$, where $\Gg_n=R(\psi_n)$ for some local homeomorphisms $\psi_n:X\to X_n$ and there are maps $\varphi_n:X_n\to X_{n+1}$ such that $\varphi_n\circ\psi_n=\psi_{n+1}$. The local homeomorphisms $\varphi_n$ induce   group homomorphisms \[\varphi_{n*}:\ZZ[X_n]\to \ZZ[X_{n+1}],\;\; \varphi_{n*}(f)(x_{n+1})=\sum_{\varphi_n(x_n)=x_{n+1}}f(x_n)\] as in \eqref{eq:pi_star}. Moreover, since each $\Gg_n$ is equivalent with $X_n$, 
we obtain \[H_0(\Gg,\ZZ)\cong \varinjlim(\ZZ[X_n],\varphi_{n*})\]  and $H_n(\Gg,\ZZ)=0$ for $n\ge 1$.

Given a $\Gg$-module $M$, we also denote  by $M$ the corresponding $\Gg_n$-module. Dualizing the bar resolution \eqref{bar} for each $\Gg_n$,  consider the tower of cochain complexes \[\cdots \to C_{n+1}\to C_n\to \cdots\to C_1\] with $C^k_n=\text{Hom}_{\Gg_n}(\ZZ[\Gg_n^{(k)}],M)$, used to compute $H^*(\Gg_n,M)$. Since the inclusion $\ZZ[\Gg_n^{(k)}]\subseteq \ZZ[\Gg_{n+1}^{(k)}]$ splits because each $\ZZ[\Gg_{n+1}^{(k)}]/\ZZ[\Gg_n^{(k)}]$ is a free abelian group, the maps $C^k_{n+1}\to C^k_n$ are onto for each $k$, and the tower satisfies the Mittag-Leffler condition (see Definition 3.5.6 in \cite{W}). Since  $\Gg=\varinjlim \Gg_n$ and $H^*(\Gg,M)$ is the cohomology of the cochain complex $C$ with \[C^k=\varprojlim C^k_n=\varprojlim\text{Hom}_{\Gg_n}(\ZZ[\Gg^{(k)}_n],M)\cong\text{Hom}_{\Gg}( \varinjlim \ZZ[\Gg_n^{(k)}],M) ,\]
a consequence of Theorem 3.5.8 in \cite{W} gives
\[0\to \varprojlim{\!}^1H^{q-1}(\Gg_n,M)\to H^q(\Gg,M)\to \varprojlim H^q(\Gg_n,M)\to 0. \]
 In particular, since $\Gg_n$ is equivalent with $X_n$, by taking $M=\Gamma_c(\underline{\ZZ})$ it follows that  
 \[H^0(\Gg,\ZZ)\cong \varprojlim(C(X_n,\ZZ),\varphi^*_{n}),\]
 \[H^1(\Gg,\ZZ)\cong \varprojlim{\!}^1H^0(\Gg_n,\ZZ)\cong \varprojlim{\!}^1(C(X_n,\ZZ), \varphi^*_{n}),\]
 where, using Remark \ref{rm:explicit_phi_star},  $\varphi^*_{n}:C(X_{n+1},\ZZ)\to C(X_n,\ZZ)$ is determined by $f\mapsto f\circ \varphi_n$ for $f\in $ Hom$_{X_{n+1}}(\ZZ[X_{n+1}], \ZZ[X_{n+1}])$ identified with $C(X_{n+1},\ZZ)$.

\begin{example} (The $UHF(p^\infty)$ groupoid)
 Let $X=\{1,2,...,p\}^\NN$ for $p\ge 2$ and let $\sigma:X\to X,\; \sigma(x_1x_2\dots)=x_2x_3\dots\;$ be the unilateral shift, which is a local homeomorphism. Then
 \[R(\sigma^n)=\{(x,y)\in X\times X:\sigma^n(x)=\sigma^n(y)\}\]
 are elementary groupoids for $n\ge 0$ and $H_0(R(\sigma^n),\ZZ)\cong C(X,\ZZ)$.  Consider the $UHF(p^\infty)$ groupoid
 \[\Ff_p=\bigcup_{n=0}^\infty R(\sigma^n).\]
 We get  $H_0(\Ff_p,\ZZ)\cong\varinjlim (C(X,\ZZ),\sigma_*)\cong \ZZ[\frac{1}{p}]$, where $\ds \sigma_*(f)(y)=\sum_{\sigma(x)=y}f(x)$.  Indeed, for $n\ge 1$ consider the  map
 \[h :C(X,\ZZ)\to  \ZZ\left[\frac{1}{p}\right],\;  h(\chi_{Z(\alpha_1\cdots\alpha_n)})=\frac{1}{p^n},\] and extended by linearity, where $Z(\alpha_1\cdots\alpha_n)$ is a cylinder set. Note that $Z(\emptyset)=X$, $h(\chi_X)=1$ and that $h$ is onto. Since  $\sigma_*(\chi_X)=p\cdot \chi_X$ and
 $\sigma_*(\chi_{Z(\alpha_1\cdots\alpha_n)})=\chi_{Z(\alpha_2\cdots\alpha_n)}$,
 it follows that \[h\circ\sigma_*=\ell_p\circ h,\] where
 $\ell_p:\ZZ[\frac{1}{p}]\to \ZZ[\frac{1}{p}]$ is multiplication by
 $p$, a bijection.  We get \[\varinjlim (C(X,\ZZ),\sigma_*)\cong
   \ZZ\left[\frac{1}{p}\right].\]

 To compute the
 cohomology, we use 
 the exact sequence
\[0\to \varprojlim{\!}^1H^{q-1}(R(\sigma^n),M)\to H^q(\Ff_p,M)\to \varprojlim H^q(R(\sigma^n),M)\to 0 \] for $M=\Gamma_c(\underline{\ZZ})$ and  the results of Example
 \ref{ex:Coh_space} to obtain
 \[
  0\to \varprojlim{}^1(H^{q-1}(X),\sigma^*)\to H^q(\Ff_p)\to
  \varprojlim (H^q(X),\sigma^*)\to 0,
 \]
 where $\sigma^*:C(X,\ZZ)\to C(X,\ZZ)$ is given by $\sigma^*(f)=f\circ
 \sigma$. 
 Therefore
 \[H^0(\Ff_p)\cong \varprojlim(C(X,\ZZ),\sigma^*)\cong \ZZ,\,
   H^1(\Ff_p)\cong \varprojlim{\!}^1(C(X,\ZZ), \sigma^*),
 \]
and $H^q(\Ff_p)=0$ for all $q\ge 2$.
Indeed, the only elements in the projective limit are the constant functions. Note that $H^1(\Ff_p)$ is uncountable.
 
 \end{example}

\begin{example}\label{transf}
For the transformation groupoid $\Gg=\Gamma \ltimes X$ associated to a  discrete group action $\Gamma \curvearrowright X$ on a Cantor set $X$, since $\Gg^{(n)}\cong\Gamma^n\times X$,
the homology chain complex for $A$  an abelian group has the form
\[0\leftarrow C_c(X, A)\leftarrow C_c(\Gamma\times X, A)\leftarrow\cdots\leftarrow C_c(\Gamma^n\times X, A)\leftarrow\cdots\]
and $H_n(\Gamma\ltimes X, A)\cong H_n(\Gamma, C(X,A))$ where  $C(X,A)$ is a $\Gamma$-module in the usual way. For $\Gamma=\ZZ$ with generator $\varphi\in\text{Homeo}(X)$, it is known that, see \cite{M12, Br}
\[H_0(\ZZ\ltimes X,A)\cong C(X,A)/\{f-f\circ\varphi^{-1}:f\in C(X,A)\},\; H_1(\ZZ\ltimes X,A)\cong A,\] and  $H_n(\ZZ\ltimes X,A)=0$ for $ n\ge 2$.

The dual complex for the transformation groupoid $\Gg=\Gamma \ltimes X$  becomes
\[0\rightarrow \text{Hom}_{\Gg}(\ZZ[X], M)\rightarrow \text{Hom}_{\Gg}(\ZZ[\Gamma\times X],M)\rightarrow \cdots\]\[\rightarrow \text{Hom}_{\Gg}(\ZZ[\Gamma^n\times X],M)\rightarrow\cdots\]
where $M$ is a $\Gg$-module.
It follows that \[H^n(\Gamma\ltimes X,M)\cong H^n(\Gamma, C(X,M)),\] the group cohomology of $\Gamma$ with coefficients in $C(X,M)$. For $\Gamma=\ZZ$, using the computation of cohomology of $\ZZ$ with coefficients from Chapter III in \cite{Br}, we get
\[H^0(\ZZ\ltimes X, M)\cong M,\; H^1(\ZZ\ltimes X, M)\cong C(X,M)/\{f-f\circ\varphi^{-1}:f\in C(X,M)\}\] and $H^n(\ZZ\ltimes X, M)=0$ for $n\ge 2$. This illustrates a particular case of Poincar\' e duality between homology and cohomology, see page 221 in \cite{Br}.

The same result for $M=\Gamma_c(\Gg^{(0)}\times \ZZ,\pi)$ is obtained by using the long exact sequence from Theorem \ref{les} if we consider $\Gg=\ZZ\ltimes X$ with cocycle $c:\Gg\to \ZZ,\; c(k,x)=k$. Then $\Gg\times_c\ZZ$ is similar to $X$ and therefore, after identifying the maps of the long exact sequence, 
\[H^0(\ZZ\ltimes X,\ZZ)\cong \ker(id-c^{*(0)})\cong \ZZ,\]\[ H^1(\ZZ\ltimes X,\ZZ)\cong \text{coker}(id-c^{*(0)})\cong C(X,\ZZ)/\{f-f\circ\varphi^{-1}:f\in C(X,\ZZ)\},\]\[ H^n(\ZZ\ltimes X,\ZZ)=0\;\text{for}\; n\ge 2.\]
\end{example}

\end{document}